\documentclass[11pt]{amsart}%
\usepackage{eurosym}
\usepackage{graphicx,amscd,color,amsmath,amsfonts,amssymb,amsthm,centernot,geometry}
\usepackage[pdftex,plainpages=false,pdfpagelabels,colorlinks=true,citecolor=black,linkcolor=black,
urlcolor=black,filecolor=black,bookmarksopen=true,unicode=true,breaklinks=true]%
{hyperref}
\usepackage[initials]{amsrefs}
\usepackage[all]{xy}
\usepackage{amsmath}
\usepackage{amsfonts}
\usepackage{amssymb}
\usepackage{graphicx}
\usepackage{mathtools}%
\providecommand{\U}[1]{\protect\rule{.1in}{.1in}}
\newtheorem{theorem}{Theorem}

\newtheorem{lemma}{Lemma}

\begin{document}
\title[ ]{On the spaceability of the set of functions in the Lebesgue space $L_p$ which are in no other $L_q$}
\author[Ara\'{u}jo]{G. Ara\'{u}jo}
\address[G. Ara\'{u}jo]{Departamento de Matem\'{a}tica \\
Universidade Estadual da Para\'{\i}ba \\
58.429-500 Campina Grande, Brazil.}
\email{gustavoaraujo@servidor.uepb.edu.br}
\author[Barbosa]{A. Barbosa}
\address[A. Barbosa]{Departamento de Matem\'atica \\
Universidade Federal da Para\'iba \\
58.051-900 Jo\~ao Pessoa, Brazil.}
\email{afsb@academico.ufpb.br}
\author[Raposo Jr.]{A. Raposo Jr.}
\address[A. Raposo Jr.]{Departamento de Matem\'atica \\
Universidade Federal do Maranh\~ao \\
65.085-580 S\~ao Lu\'is, Brazil.}
\email{anselmo.junior@ufma.br}
\author[Ribeiro]{G. Ribeiro}
\address[G. Ribeiro]{Departamento de Matem\'atica \\
Universidade Federal da Para\'iba \\
58.051-900 Jo\~ao Pessoa, Brazil.}
\email{geivison.ribeiro@academico.ufpb.br}
\thanks{\textit{Conflict of interest.} The authors declare that they have no conflict of interest.}
\thanks{The first author was partially supported by Grant 3024/2021, Para\'iba State Research Foundation (FAPESQ). The fourth author was supported by Grant 2022/1962, Para\'iba State Research Foundation (FAPESQ). This study was financed in part by the Coordenação de Aperfeiçoamento de Pessoal de Nível Superior - Brasil (CAPES) - Finance Code 001}

\begin{abstract}
In this note we prove that, for $p>0$, $L_{p}[0,1]\smallsetminus\bigcup_{q\in(p,\infty)}L_{q}[0,1]$ is $(\alpha,\mathfrak{c})$-spaceable if, and only if, $\alpha<\aleph_{0}$. Such
a problem first appears in [V. F\'avaro, D. Pellegrino, D. Tomaz, Bull. Braz.
Math. Soc. \textbf{51} (2020) 27-46], where the authors get the
$(1,\mathfrak{c})$-spaceability of $L_{p}[0,1]\smallsetminus\bigcup_{q\in(p,\infty)}%
L_{q}[0,1]$ for $p>0$. The definitive answer to this problem continued to be
sought by other authors, and some partial answers were obtained. The
veracity of this result was expected, as a similar result is known for
sequence spaces.
\end{abstract}

\maketitle

%\tableofcontents

\section{Introduction and motivation}

From now on all vector spaces are considered over a fixed scalar field
$\mathbb{K}$ which can be either $\mathbb{R}$ or $\mathbb{C}$. For any set $X$
we shall denote by $\operatorname*{card}(X)$ the cardinality of $X$; we also
define $\mathfrak{c}=\operatorname*{card}(\mathbb{R})$ and $\aleph
_{0}=\operatorname*{card}(\mathbb{N})$.
%If $X$ is a vector space and $Y$ is a vector subspace of $X$, we denote the infinite algebraic codimension of $Y$ by $\mathrm{codim}(Y)$.

If $E$ is a vector space, $\beta\leq\dim(E)$ is a cardinal number and
$A\subset E$, then $A$ is said to be $\beta$\textit{-lineable} if there exists
a vector space $F_{\beta}$ with $\dim(F_{\beta})=\beta$ and $F_{\beta
}\smallsetminus\{0\}\subset A$. If $E$ is, in addition, endowed with a
topology, then $A$ is called $\beta$\textit{-spaceable} if $A\cup\{0\}$
contains a closed $\beta$-dimensional linear subspace of $E$ (see \cite{AGS}).
Also, if $\alpha$ is another cardinal number, with $\alpha\leq\beta$, then $A$
is said to be $(\alpha,\beta)$\textit{-spaceable} if it is $\alpha$-lineable
and for every subspace $F_{\alpha}\subset E$ with $F_{\alpha}\subset
A\cup\{0\}$ and $\dim(F_{\alpha})=\alpha$, there is closed subspace $F_{\beta
}\subset E$ with $\dim(F_{\beta})=\beta$ and $F_{\alpha}\subset F_{\beta
}\subset A\cup\{0\}$ (see \cite{FPT}).

The concept of \textit{lineability} was coined by V. I. Gurariy in the early
2000's and it first appeared in print in \cites{GQ,AGS}. V. I. Gurariy's
interest in linear structures in generally non-linear settings dates as far
back as 1966 (see \cite{gurariy1966}). The study of large vector structures in
sets of real and complex functions has attracted many mathematicians in the
last decade. For example, for $0<p\leq\infty$, in 2008 \cite{mpps2008}, 2009
\cite{agps2009}, 2010 \cite{bernalppp}, 2011 \cite{BDFP}, 2012 \cite{BFPS,bc2}%
, 2020 \cite{FPT} and 2021 \cite{DR}, Aron, Bernal-Gonz\'{a}lez, Botelho,
Diniz, F\'{a}varo, Garc\'{\i}a-Pacheco, Mu\~{n}oz-Fern\'{a}ndez,
Ord\'{o}\~{n}ez-Cabrera, Palmberg, Pellegrino, P\'{e}rez-Garc\'{\i}a, Puglisi,
Raposo Jr., Seoane-Sep\'{u}lveda and Tomaz proved a set of interesting
spaceability results concerning the vector space $L_{p}(\mu,\Omega)$ of all
(Lebesgue classes of) measurable functions $f\colon\Omega\rightarrow
\mathbb{K}$ such that
\[%
\begin{cases}
|f|^{p}\text{ is }\mu\text{-integrable on }\Omega\text{,} & \text{for
}0<p<\infty\text{,}\\
f\text{ is }\mu\text{-essentially bounded in }\Omega\text{,} & \text{for
}p=\infty\text{,}
\end{cases}
\]
(here $(\Omega,\mathcal{M},\mu)$ is a measure space). Below we will recall
some of the results mentioned above. Before, recall that $L_{p}(\mu,\Omega)$,
$0<p<\infty$, becomes a Banach space (quasi-Banach if $p<1$) under the norm
($p$-norm if $p<1$)
\[
\Vert f\Vert_{p}=\left(  \int_{\Omega}|f|^{p}d\mu\right)  ^{\frac{1}{p}%
}\text{.}%
\]
If $p=\infty$, $L_{\infty}(\mu,\Omega)$ becomes a Banach space under the norm
\[
\Vert f\Vert=\inf\{M>0:|f|\leq M\ \mu\text{-almost everywhere in }X\}\text{.}%
\]
As usual, if $\Omega=I\subset\mathbb{R}$ and $\mu=$ the Lebesgue measure, we
denote $L_{p}(\mu,I)$ by $L_{p}(I)$, and if $\Omega=\mathbb{N}$ and $\mu=$ the
counting measure, we denote $L_{p}(\mu,\mathbb{N})$ by $\ell_{p}$.

%For instance, for $p>0$, the $L_p[0,1]$ space is the quotient of the space
%\[
%\mathcal{L}_p[0,1]=\left\{f:[0,1]\rightarrow\mathbb{K}:f\text{ is measurable and} \int_{[0,1]}|f(x)|^p dx<\infty\right\}
%\]
%by the equivalence relation
%\[
%f\sim g \Leftrightarrow \mu(\{x\in[0,1]:\ f(x)\neq g(x)\})=0,
%\]
%equipped with the norm ($p$-norm if $p<1$)
%\begin{eqnarray*}
%\|f\|_p=\left(\int_{[0,1]}|f(x)|^p dx\right)^\frac{1}{p},
%\end{eqnarray*}
%and $\ell_p$ denotes the Banach space (quasi-Banach if $p<1$) of the sequences $(x_k)_{k\in\mathbb{N}}$ such that
%\[
%\left\|(x_k)_{k\in\mathbb{N}}\right\|_p=\left(\sum_{k=1}\infty|x_k|^p\right)^{\frac{1}{p}}<\infty.
%\]

The following results are well known:

\begin{theorem}
[\cite{FPT}]\label{222} For all $0<p\leq\infty$ the set
\[
\ell_{p}\smallsetminus\bigcup_{q\in\left(  0,p\right)  }\ell_{q}%
\]
is $(\alpha,\mathfrak{c})$-spaceable in $\ell_{p}$ if, and only if,
$\alpha<\aleph_{0}$.
\end{theorem}

\begin{theorem}
[\cite{BFPS}]\label{boteee} The set
\begin{equation}
L_{p}[0,1]\smallsetminus\bigcup_{q\in\left(  p,\infty\right)  }L_{q}[0,1]
\label{lp-lq}%
\end{equation}
is spaceable for all $0<p<\infty$.
\end{theorem}

%In a more general context, the authors of \cite{bc2} achieved, among others, the following results on the spaceability of the difference of $L_{p}%(\mu,\Omega)$ spaces.

%\begin{theorem}
%[\cite{bc2}]\label{t31} Assume that $(\Omega,\mathcal{M},\mu)$ is a measure space. Consider the following conditions:

%\begin{itemize}
%\item[($\alpha$)] $\inf\{\mu(A):\ A\in\mathcal{M},\ \mu(A)>0\}=0$;

%\item[($\beta$)] $\sup\{\mu(A):\ A\in\mathcal{M},\ \mu(A)<\infty\}=\infty$.
%\end{itemize}

%Then:

%\begin{itemize}
%\item[(1)] if $1\leq p<\infty$, $L_{p}(\mu,\Omega)\smallsetminus\bigcup
%_{q\in\left(  p,\infty\right)  }L_{q}(\mu,\Omega)$ is spaceable if and only if
%$\left(  \alpha\right)  $ holds;

%\item[(2)] if $1<p\leq\infty$, $L_{p}(\mu,\Omega)\smallsetminus\bigcup
%_{q\in\lbrack1,p)}L_{q}(\mu,\Omega)$ is spaceable if and only if $\left(
%\beta\right)  $ holds;

%\item[(3)] if $1<p<\infty$, $L_{p}(\mu,\Omega)\smallsetminus\bigcup _{q\in\lbrack1,\infty)\smallsetminus\{p\}}L_{q}(\mu,\Omega)$ is spaceable if and only if both $\left(  \alpha\right)  $ and $\left(  \beta\right)  $holds;
%\item[(4)] if $1< p<\infty$ and $L_p(\Omega)$ is separable, then $L_p(\Omega)\smallsetminus\cup_{q\in[1,p)}L_q(\Omega)$ is maximal-dimension dense-lineable if and only if ($\beta$) holds.

%\end{itemize}
%\end{theorem}

The proof of Theorem \ref{boteee} does not guarantee that $L_{p}%
[0,1]\smallsetminus\bigcup_{q\in\left(  p,\infty\right)  }L_{q}[0,1]$ is
$(\alpha,\mathfrak{c})$-spaceable for some cardinal $\alpha>0$. A result by
F\'{a}varo et al. \cite{FPT} shows that this is true for $\alpha=1$ and in the
same article they ask about the $(\alpha,\mathfrak{c})$-spaceability of the
set in \eqref{lp-lq} for a cardinal $1<\alpha<\mathfrak{c}$ (this same issue
is again highlighted in \cite{FPR}). Later on, in \cite{fprs}, F\'{a}varo et
al. proved that the set in \eqref{lp-lq} is not $(\alpha,\beta)$-spaceable for
$\alpha\geq\aleph_{0}$, regardless of the cardinal number $\beta$.

Summarizing all the information above, we have the following question:

\medskip

\begin{center}
\begin{minipage}{14cm}
For $0<p<\infty$ and $2\leq\alpha<\aleph_{0}$, is the set $L_{p}[0,1]\smallsetminus\bigcup_{q\in\left(  p,\infty\right)  }L_{q}[0,1]$ $(\alpha,\mathfrak{c})$-spaceable?
\end{minipage}
\end{center}

\medskip

In view of Theorem \ref{222}, many authors conjectured the veracity of this question. In this note, using a slightly different technique than the one usually used in this type of problem, namely the \textit{mother vector technique}, we answer the above question.

\section{Main result}

\begin{theorem}
\label{t27} For all $0< p<\infty$ the set
\[
L_{p}[0,1]\smallsetminus\bigcup_{q\in\left(  p,\infty\right)  }L_{q}[0,1]
\]
is $(\alpha,\mathfrak{c})$-spaceable in $L_{p}[0,1]$ if, and only if,
$\alpha<\aleph_{0}$.
%is $(\alpha,\beta)$-spaceable in $L_{p}[0,1]$ if, and only if, $\alpha \leq \beta$ with $\alpha \in \mathbb{N}$ and $\beta \in \mathbb{N} \cup \{\mathfrak{c}\}.$
\end{theorem}

\begin{proof}
From the previous discussion the question remains open only for $2\leq
\alpha<\aleph_{0}$.

Let $g_{1},\ldots,g_{n}\in L_{p}[0,1]$ be linearly independent normalized
vectors so that%
\[
\operatorname*{span}\{g_{1},\ldots,g_{n}\}\smallsetminus\{0\}\subset
L_{p}[0,1]\smallsetminus\bigcup_{q\in\left(  p,\infty\right)  }L_{q}%
[0,1]\text{.}%
\]
Let us consider the representation of the semi-open interval  $(0,1]$ as the following disjoint union
\[
(0,1]=\bigcup_{k=1}^{\infty}I_{k},
\]
where $I_{k}:=\left(  \frac{1}{k+1},\frac{1}{k}\right]$. Let us fix  $k\in\mathbb{N}$. Since $\bigcup_{q\in\left(  p,\infty\right)  }L_{q}%
 (I_{k})$ is a vector subspace of $L_{p}(I_{k})$ and $\bigcup_{q\in\left(  p,\infty\right)
 }L_{q}(I_{k})$ has infinite codimension (see \cite[Theorem 4.4]{Bernal}), we can take an infinite dimensional subspace $V_{k}$ of
$L_{p}(I_{k})$ so that
\[
L_{p}(I_{k})=V_{k}\oplus\bigcup_{q\in\left(  p,\infty\right)  }L_{q}%
(I_{k})\text{.}%
\]
Now, consider the canonical projection $P_{k}\colon L_{p}(I_{k})\rightarrow V_{k}$
of $L_{p}(I_{k})$ onto $V_{k}$ and let%
\[
\tilde{f}_{k}\in V_{k}\smallsetminus P_{k}\left(  \operatorname*{span}%
\{g_{1}|_{I_{k}},\ldots,g_{n}|_{I_{k}}\}\right)
\]
with $\Vert\tilde{f}_{k}\Vert_{p}=1$. Let us prove that, for all $a_{1},\ldots
,a_{n}\in\mathbb{K}$,
\begin{equation}
\tilde{f}_{k}+\sum_{i=1}^{n}a_{i}g_{i}|_{I_{k}}\notin\bigcup_{q\in\left(
p,\infty\right)  }L_{q}(I_{k}) \label{eq}.
\end{equation}
In fact, if there exists $a_{1},\ldots
,a_{n}\in\mathbb{K}$ such that $\tilde{f}_{k}+\sum_{i=1}^{n}a_{i}g_{i}|_{I_{k}}\in\bigcup
_{q\in\left(  p,\infty\right)  }L_{q}(I_{k})$, since
\[
\tilde{f}_{k}+\sum_{i=1}^{n}a_{i}g_{i}|_{I_{k}}=\tilde{f}_{k}+P_{k}\left(
\sum_{i=1}^{n}a_{i}g_{i}|_{I_{k}}\right)  +\left(  -P_{k}\left(  \sum
_{i=1}^{n}a_{i}g_{i}|_{I_{k}}\right)  +\sum_{i=1}^{n}a_{i}g_{i}|_{I_{k}%
}\right)  \text{,}%
\]
we would conclude that $\tilde{f}_{k}+P_{k}\left(  \sum_{i=1}^{n}a_{i}%
g_{i}|_{I_{k}}\right)  =0$ and, hence, $\tilde{f}_{k}\in P_{k}\left(
\operatorname*{span}\{g_{1}|_{I_{k}},\ldots,g_{n}|_{I_{k}}\}\right)  $, which 
we know doesn't happen.

Define $\tilde{p}=1$ if $p \geq 1$ and $\tilde{p}=p$ if $0<p<1$. Furthermore, consider $f_{k}\in L_{p}[0,1]\smallsetminus\bigcup_{q\in\left(  p,\infty\right)
}L_{q}[0,1]$, where
\[
f_{k}=
\begin{cases}
0 & \text{in } [0,1]\smallsetminus I_{k}\\
\tilde{f}_{k} & \text{in } I_{k}.
\end{cases}
\]
For $\left(a_{i}\right)_{i=1%
}^\infty\in\ell_{\tilde{p}}$,
\[
\Vert a_{1}g_{1}\Vert_{p}^{\tilde{p}}+\cdots+\Vert a_{n}g_{n}\Vert_{p}^{\tilde{p}}+\sum_{i=n+1}%
^{\infty}\Vert a_{i}f_{i-n}\Vert_{p}^{\tilde{p}}=\sum_{i=1}^{\infty}|a_{i}|^{\tilde{p}}<\infty\text{.}%
\]
Since $L_{p}[0,1]$ is a Banach space for $p\geq1$ and a quasi Banach space for $0<p<1$, it follows that
$a_{1}g_{1}+\cdots+a_{n}g_{n}+\sum_{i=n+1}^{\infty}a_{i}f_{i-n}\in L_{p}[0,1]$. Therefore we can define the operator
\[
T\colon\ell_{\tilde{p}}\rightarrow L_{p}[0,1]\text{, \ \ }\ \ T\left(  \left(a_{i}\right)_{i=1%
}^\infty\right)  =a_{1}g_{1}+\cdots+a_{n}g_{n}+\sum_{i=n+1}^{\infty
}a_{i}f_{i-n}.
\]
For an arbitrary function $f\colon X\rightarrow\mathbb{K}$ whose
domain is an arbitrary set $X$, let $\operatorname{supp}\left(  f\right)
=\{x\in X:f(x)\neq0\}$. Since $\operatorname{supp}(f_{i})\cap
\operatorname{supp}(f_{j})=\varnothing$ for $i\neq j$, we can conclude that
$T(\ell_{\tilde{p}})$ has infinite dimension.

Below we will show that there exists a positive integer $m_{0}$ such that
\[
\{g_{1}|_{\cup_{i=1}^{m_0}I_{i}},\ldots,g_{n}|_{\cup_{i=1}^{m_0}I_{i}}%
,f_{1}|_{\cup_{i=1}^{m_0}I_{i}},\ldots,f_{m_0}|_{\cup_{i=1}^{m_0}I_{i}}\}
\]
is a linearly independent set in $L_{p}(\bigcup_{i=1}^{m_0}I_{i})$. We first need to prove the following lemma:

\begin{lemma}
\label{lema} There exist a positive integer $m_{1}$ such that
\[
\{g_{1}|_{\cup_{i=1}^{m_1}I_{i}},\ldots,g_{n}|_{\cup_{i=1}^{m_1}I_{i}}\}
\]
is a linearly independent set in $L_{p}(\bigcup_{i=1}^{m_1}I_{i})$.
\end{lemma}

\begin{proof}
[Proof of Lemma \ref{lema}] Fix $j\in\{1,\ldots,n\}$. Since $g_{j}|_{\cup
_{i=1}^{m}I_{i}}\overset{m\rightarrow\infty}{\longrightarrow}g_{j}$ in
$L_{p}[0,1]$, we have $g_{j}|_{\cup_{i=1}^{m}I_{i}}\neq0$ for all large enough $m$.
By contradiction, suppose there is not a positive integer $m_{1}$ such that
$\{g_{1}|_{\cup_{i=1}^{m_1}I_{i}},\ldots,g_{n}|_{\cup_{i=1}^{m_1}I_{i}}\}$
is linearly independent in $L_{p}(\bigcup_{i=1}^{m_1}I_{i})$. Thus, the set $\{g_{1}|_{\cup_{i=1}^{m}I_{i}},\ldots
,g_{n}|_{\cup_{i=1}^{m}I_{i}}\}$ is linearly dependent on $L_{p}(\bigcup
_{i=1}^{m}I_{i})$ for all $m\in\mathbb{N}$. For each $m\in\mathbb{N}%
$, let $\{g_{1(m)}|_{\cup_{i=1}^{m}I_{i}},\ldots,g_{r(m)}|_{\cup
_{i=1}^{m}I_{i}}\}$ be a smaller linearly dependent subset of $\{g_{1}%
|_{\cup_{i=1}^{m}I_{i}},\ldots,g_{n}|_{\cup_{i=1}^{m}I_{i}}\}$ and define
$\varphi\colon\mathbb{N}\rightarrow\mathcal{P}(\{1,\ldots,n\})$ by
$\varphi(m)=\{1(m),\ldots,r(m)\}$, where $\mathcal{P}(\{1,\ldots,n\})$ is the
set of all subsets of $\{1,\ldots,n\}$. Since $\operatorname*{card}%
(\mathcal{P}(\{1,\ldots,n\}))<\operatorname*{card}(\mathbb{N})=\aleph_{0}%
$, there is $\{j_{1},\ldots,j_{r}\}\in\varphi(\mathbb{N})$ such that
$\operatorname*{card}(\varphi^{-1}(\{j_{1},\ldots,j_{r}\}))=\aleph_{0}$.
Define $\mathbb{N}^{\prime}:=\varphi^{-1}(\{j_{1},\ldots,j_{r}%
\})\subset\mathbb{N}$ and note that
\[
\{g_{1(m)}|_{\cup_{i=1}^{m}I_{i}},\ldots,g_{r(m)}|_{\cup_{i=1}^{m}I_{i}%
}\}=\{g_{j_{1}}|_{\cup_{i=1}^{m}I_{i}},\ldots,g_{j_{r}}|_{\cup_{i=1}^{m}I_{i}%
}\}\text{.}%
\]
%Para cada $m\in\mathbb{N}_0$ considere $\{g_{j_1}|_{\cup_{i=1}^mI_i}, \ldots, g_{j_{r_m}}|_{\cup_{i=1}^mI_i}\}$ um menor subconjunto linearmente dependente de $\{g_1|_{\cup_{i=1}^mI_i}, \ldots, g_n|_{\cup_{i=1}^mI_i}\}$, onde $r_m\leq n$. Como o conjunto das partes de $\{1,\ldots,n\}$ é finito, devem existir $\{j'_1,\ldots,j'_{r}\}\subset\{1,\ldots,n\}$ e $\mathbb{N}'_0\subset\mathbb{N}_0$ infinito, tais que $\{g_{j_1}|_{\cup_{i=1}^mI_i}, \ldots, g_{j_{r_m}}|_{\cup_{i=1}^mI_i}\}=\{g_{j'_1}|_{\cup_{i=1}^mI_i}, \ldots, g_{j'_r}|_{\cup_{i=1}^mI_i}\}$ para todo $m\in\mathbb{N}'_0$.
Thus, if $m,\tilde{m}\in\mathbb{N}^{\prime}$ are such that $m<\tilde{m}$, then there are $b_{1}%
,\ldots,b_{r-1},\tilde{b}_{1},\ldots,\tilde{b}_{r-1}\in\mathbb{K}$ so that
\[
g_{j_{r}}|_{\cup_{i=1}^{m}I_{i}}=b_{1}g_{j_{1}}|_{\cup_{i=1}^{m}I_{i}}%
+\cdots+b_{r-1}g_{j_{r-1}}|_{\cup_{i=1}^{m}I_{i}}%
\]
and
\begin{equation}
g_{j_{r}}|_{\cup_{i=1}^{\tilde{m}}I_{i}}=\tilde{b}_{1}g_{j_{1}}|_{\cup
_{i=1}^{\tilde{m}}I_{i}}+\cdots+\tilde{b}_{r-1}g_{j_{r-1}}|_{\cup
_{i=1}^{\tilde{m}}I_{i}}\text{.} \label{eq3}%
\end{equation}
Restricting \eqref{eq3} to $\bigcup_{i=1}^{m}I_{i}$ we get
\begin{align*}
\tilde{b}_{1}g_{j_{1}}|_{\cup_{i=1}^{m}I_{i}}+\cdots+\tilde{b}_{r-1}%
g_{j_{r-1}}|_{\cup_{i=1}^{m}I_{i}}  &  =g_{j_{r}}|_{\cup_{i=1}^{m}I_{i}}\\
&  =b_{1}g_{j_{1}}|_{\cup_{i=1}^{m}I_{i}}+\cdots+b_{r-1}g_{j_{r-1}}%
|_{\cup_{i=1}^{m}I_{i}}%
\end{align*}
and consequently
\[
(\tilde{b}_{1}-b_{1})g_{j_{1}}|_{\cup_{i=1}^{m}I_{i}}+\cdots+(\tilde{b}%
_{r-1}-b_{r-1})g_{j_{r-1}}|_{\cup_{i=1}^{m}I_{i}}=0\text{.}%
\]
Since $\{g_{j_{1}}|_{\cup_{i=1}^{m}I_{i}},\ldots,g_{j_{r}}|_{\cup_{i=1}%
^{m}I_{i}}\}$ is a smaller linearly dependent subset of $\{g_{1}|_{\cup
_{i=1}^{m}I_{i}},\ldots,g_{n}|_{\cup_{i=1}^{m}I_{i}}\}$ we can conclude that
$\tilde{b}_{k}=b_{k}$, $k=1,\ldots,r-1$. 
Since $m\in\mathbb{N}^\prime$ is arbitrary, we obtain
\[
g_{j_{r}}|_{\cup_{i=1}^{m}I_{i}}=b_{1}g_{j_{1}}|_{\cup_{i=1}^{m}I_{i}}%
+\cdots+b_{r-1}g_{j_{r-1}}|_{\cup_{i=1}^{m}I_{i}}%
\]
for all $m\in\mathbb{N}^\prime$. 

Therefore
\begin{align*}
g_{j_{r}}  &=\lim\limits_{m\in\mathbb{N}^{\prime}} g_{j_{r}}|_{\cup_{i=1}^{m}I_{i}}\\
&=\lim\limits_{m\in\mathbb{N}^{\prime}}\left(  b_{1}g_{j_{1}}%
|_{\cup_{i=1}^{m}I_{i}}+\cdots+b_{r-1}g_{j_{r-1}}|_{\cup_{i=1}^{m}I_{i}%
}\right) \\
&  =b_{1}g_{j_{1}}+\cdots+b_{r-1}g_{j_{r-1}}\text{,}%
\end{align*}
which is contrary to the fact that $\{g_{1},\ldots,g_{n}\}$ is linearly independent.
\end{proof}

Let us return to the proof of Theorem \ref{t27}. Let us prove that the set
\[
\{g_{1}|_{\cup_{i=1}^{m}I_{i}},\ldots,g_{n}|_{\cup_{i=1}^{m}I_{i}}%
,f_{1}|_{\cup_{i=1}^{m}I_{i}},\ldots,f_{m}|_{\cup_{i=1}^{m}I_{i}}\}
\]
is linearly independet in $L_{p}(\bigcup_{i=1}^{m}I_{i})$ for all $m\geq
m_{0}$, where $$m_{0}=\min\left\{m_{1}:\{g_{1}|_{\cup_{i=1}^{m}I_{i}},\ldots
,g_{n}|_{\cup_{i=1}^{m}I_{i}}\} \text{ is linearly independent in } L_{p}\left(\bigcup_{i=1}^{m}I_{i}\right)\right\}.$$

Given $m\geq m_0$, let $b_{1},\ldots,b_{n},b_{n+1},\ldots,b_{n+m}\in\mathbb{K}$ such that
\[
b_{1}g_{1}|_{\cup_{i=1}^{m}I_{i}}+\cdots+b_{n}g_{n}|_{\cup_{i=1}^{m}I_{i}%
}+b_{n+1}f_{1}|_{\cup_{i=1}^{m}I_{i}}+\cdots+b_{n+m}f_{m}|_{\cup_{i=1}%
^{m}I_{i}}=0\text{,}%
\]
i.e.,
\begin{equation}
b_{1}g_{1}|_{\cup_{i=1}^{m}I_{i}}+\cdots+b_{n}g_{n}|_{\cup_{i=1}^{m}I_{i}%
}=-b_{n+1}f_{1}|_{\cup_{i=1}^{m}I_{i}}-\cdots-b_{n+m}f_{m}|_{\cup_{i=1}%
^{m}I_{i}}\text{.} \label{eq5}%
\end{equation}
Restricting the equality in \eqref{eq5} to $I_{j}$, $j=1,\ldots,m$, we have
\[
b_{1}g_{1}|_{I_{j}}+\cdots+b_{n}g_{n}|_{I_{j}}=-b_{n+j}\tilde{f}_{j}\text{,}
\]
i.e., $-b_{n+j}\tilde{f}_{j}=P_{j}\left(  b_{1}g_{1}|_{I_{j}}+\cdots+b_{n}%
g_{n}|_{I_{j}}\right)  \in P_{j}\left(  \operatorname*{span}\{g_{1}|_{I_{j}%
},\ldots,g_{n}|_{I_{j}}\}\right)  $, and we can conclude that $b_{n+j}=0$. From \eqref{eq5} we have
\[
b_{1}g_{1}|_{\cup_{i=1}^{m}I_{i}}+\cdots+b_{n}g_{n}|_{\cup_{i=1}^{m}I_{i}%
}=0\text{,}%
\]
and from the Lemma \ref{lema} we obtain $b_{1}=\ldots=b_n=0$.

Now let us see that
\[
\overline{T(\ell_{\tilde{p}})}\smallsetminus\{0\}\subset L_{p}[0,1]\smallsetminus
\bigcup_{q\in\left(  p,\infty\right)  }L_{q}[0,1]\text{.}%
\]
Indeed, given $h\in\overline{T(\ell_{\tilde{p}})}\smallsetminus\{0\}$, let $(a_{i}^{(k)})_{i=1}^\infty\in\ell_{\tilde{p}}$ ($k\in\mathbb{N}$) such that
\[
T\left(  (a_{i}^{(k)})_{i=1%
}^\infty\right)  \overset{k\rightarrow\infty
}{\longrightarrow}h\ \text{in }L_{p}[0,1]\text{.}%
\]
Observe that $T\left(  (a_{i}^{(k)})_{i=1%
}^\infty\right)  |_{I}\overset{k\rightarrow\infty}{\longrightarrow}h|_{I}$ in
$L_{p}(I)$ for any subinterval $I$ of $[0,1]$. In order to go further, the strategy shall be to prove that there is a sequence of scalars $(a_{i})_{i\in\mathbb{N}}$ such that
\[
a_{1}g_{1}+\cdots+a_{n}g_{n}+\sum_{i=1}^{\infty}a_{n+i}f_{i}=h.
\]
In fact, for a fixed $m\geq m_0$, note that
\begin{align*}
a_{1}^{(k)}g_{1}|_{\cup_{i=1}^{m}I_{i}}+\cdots+a_{n}^{(k)}g_{n}|_{\cup_{i=1}^{m}I_{i}}   +a_{n+1}^{(k)}f_{1}|_{\cup_{i=1}^{m}I_{i}}+\cdots+ \ & a_{n+m}^{(k)}f_{m}|_{\cup_{i=1}^{m}I_{i}}\\
& =T\left(  \left(a_{i}^{(k)}\right)_{i=1}^\infty\right)  |_{\cup_{i=1}^{m}I_{i}}\overset{k\rightarrow\infty}{\longrightarrow}h|_{\cup_{i=1}^{m}I_{i}}\text{,}
\end{align*}
and that $\mathrm{span}\left\{g_{1}|_{\cup_{i=1}^{m}I_{i}},\ldots,g_{n}|_{\cup_{i=1}^{m}I_{i}},f_{1}|_{\cup_{i=1}^{m}I_{i}},\ldots,f_{m}|_{\cup_{i=1}^{m}I_{i}}\right\}$ is finite dimensional on $L_{p}(\bigcup_{i=1}^{m}I_{i})$. Since every finite-dimensional subspace of a topological vector space is closed, there are scalars $a_1(m),\ldots, a_{n+m}(m)$ such that
\begin{equation}\label{beq5}
h|_{\cup_{i=1}^{m}I_{i}}=a_{1}(m)g_{1}|_{\cup_{i=1}^{m}I_{i}}+\cdots+a_{n}(m)g_{n}|_{\cup_{i=1}^{m}I_{i}}+a_{n+1}(m)f_{1}|_{\cup_{i=1}^{m}I_{i}}+\cdots+a_{n+m}(m)f_{m}|_{\cup_{i=1}^{m}I_{i}},
\end{equation}
Obviously the same reasoning can be applied to $\tilde{m}>m$ and therefore
\begin{equation}\label{beq6}
h|_{\cup_{i=1}^{\tilde{m}}I_{i}}=a_{1}(\tilde{m})g_{1}|_{\cup_{i=1}^{\tilde{m}}I_{i}}+\cdots+a_{n}(\tilde{m})g_{n}|_{\cup_{i=1}^{\tilde{m}}I_{i}}+a_{n+1}(\tilde{m})f_{1}|_{\cup_{i=1}^{\tilde{m}}I_{i}}+\cdots+a_{n+\tilde{m}}(\tilde{m})f_{\tilde{m}}|_{\cup_{i=1}^{\tilde{m}}I_{i}}.
\end{equation}
Restricting \eqref{beq6} to $\bigcup_{i=1}^{m}I_{i}$ and comparing with \eqref{beq5} we get
\begin{align*}
&a_{1}(\tilde{m})g_{1}|_{\cup_{i=1}^{m}I_{i}}+\cdots+a_{n}(\tilde{m})g_{n}|_{\cup_{i=1}^{m}I_{i}}+a_{n+1}(\tilde{m})f_{1}|_{\cup_{i=1}^{m}I_{i}}+\cdots+a_{n+m}(\tilde{m})f_{m}|_{\cup_{i=1}^{m}I_{i}}\\
&=h|_{\cup_{i=1}^{m}I_{i}}\\
&=a_{1}(m)g_{1}|_{\cup_{i=1}^{m}I_{i}}+\cdots+a_{n}(m)g_{n}|_{\cup_{i=1}^{m}I_{i}}+a_{n+1}(m)f_{1}|_{\cup_{i=1}^{m}I_{i}}+\cdots+a_{n+m}(m)f_{m}|_{\cup_{i=1}^{m}I_{i}}.
\end{align*}
Since the set $\left\{g_{1}|_{\cup_{i=1}^{m}I_{i}},\ldots,g_{n}|_{\cup_{i=1}^{m}I_{i}},f_{1}|_{\cup_{i=1}^{m}I_{i}},\ldots,f_{m}|_{\cup_{i=1}^{m}I_{i}}\right\}$ is linearly independent, we obtain $a_j(m)=a_j(\tilde{m})$ for every $j=1,\ldots,n+m$. Thus we conclude that there is a sequence of scalars $(a_{i})_{i=1}^{\infty}
$ such that
\begin{align*}
\left(  a_{1}g_{1}+\cdots+a_{n}g_{n}+\sum_{i=1}^{\infty}a_{n+i}f_{i}\right)
|_{\cup_{i=1}^{m}I_{i}}  &  =(a_{1}g_{1}+\cdots+a_{n}g_{n})|_{\cup_{i=1}%
^{m}I_{i}}+\left(  \sum_{i=1}^{m}a_{n+i}f_{i}\right)  |_{\cup_{i=1}^{m}I_{i}%
}\\
&  =h|_{\cup_{i=1}^{m}I_{i}}%
\end{align*}
and so we finally have
\[
a_{1}g_{1}+\cdots+a_{n}g_{n}+\sum_{i=1}^{\infty}a_{n+i}f_{i}=h.
\]

Since $h\neq0$, it follows that $\left(a_{i}\right)_{i=1}^\infty\neq0$. Therefore, if
$a_{n+i}=0$ for all $i\in\mathbb{N}$, we have
\[
h=a_{1}g_{1}+\cdots+a_{n}g_{n}\in\operatorname*{span}\{g_{1},\ldots
,g_{n}\}\smallsetminus\{0\}\subset L_{p}[0,1]\smallsetminus\bigcup
_{q\in\left(  p,\infty\right)  }L_{q}[0,1]\text{.}%
\]

On the other hand, if $a_{n+i}\neq0$ for some $i\in\mathbb{N}$, from \eqref{eq} we obtain
\[
\frac{1}{a_{n+i}}h|_{I_{i}}=\tilde{f}_{i}+\frac{1}{a_{n+i}}\left(  a_{1}%
g_{1}+\cdots+a_{n}g_{n}\right)  |_{I_{i}}\notin\bigcup_{q\in\left(
p,\infty\right)  }L_{q}(I_{i}).
\]
Consequently, $h\notin\bigcup_{q\in\left(  p,\infty\right)  }L_{q}[0,1]$ and the result is done.
\end{proof}


\begin{thebibliography}{99}                                             \bibitem{agps2009} R. M. Aron, F. J. Garc\'{\i}a-Pacheco, D. P\'{e}rez-Garc\'{\i}a, J. B. Seoane-Sep\'{u}lveda, On dense-lineability of sets of functions on $\mathbb{R}$, Topology \textbf{48} (2009), no. 2-4, 149-156.

\bibitem{AGS} R. M. Aron, V. I. Gurariy, J. B. Seoane-Sep\'{u}lveda, Lineability and spaceability of sets of functions on $\mathbb{R}$, Proc. Amer. Math. Soc. \textbf{133} (2005), no. 3, 795-803.

%\bibitem{AP} G. Araújo, D. Pellegrino, Optimal estimates for summing multilinear operators, Linear Multilinear Algebra \textbf{65} (2017), no. 5, 930-942.

\bibitem{bernalppp} L. Bernal-Gonz\'alez, Algebraic genericity of strict order integrability, Studia Math. \textbf{199} (2010), 279-293.

%\bibitem{bc} L. Bernal-González, M. O. Cabrera, Lineability criteria, with applications, J. Funct. Anal. \textbf{266} (2014), 3997-4025.

%\bibitem{bbfp} C. S. Barroso, G. Botelho, V. V. Fávaro, D. Pellegrino, Lineability and spaceability for the weak form of Peano's theorem and vector-valued sequence spaces, Proceedings of the American Mathematical Society \textbf{141} (2013), no. 6, 1913-1923.

\bibitem{Bernal} L. Bernal-Gonz\'{a}lez, M. O. Cabrera, Lineability criteria, with applications, J. Funct. Anal. \textbf{266} (2014), 3997-4025.

\bibitem{bc2} L. Bernal-Gonz\'{a}lez, M. O. Cabrera, Spaceability of strict order integrability, J. Math. Anal. Appl. \textbf{385} (2012), 303-309.

%\bibitem{BCP} G. Botelho, D. Cariello, D. Pellegrino, Maximal spaceability in sequence spaces, Linear Algebra Appl. \textbf{437} (2012), 2978–2985.

\bibitem{BDFP} G. Botelho, D. Diniz, V. F\'avaro, D. Pellegrino, Spaceability in Banach and quasi-Banach sequence spaces, Linear Algebra Appl. \textbf{434} (2011), 1255-1260.

\bibitem{BFPS} G. Botelho, V. F\'{a}varo, D. Pellegrino, J. B. Seoane-Sep\'{u}lveda, $L_{p}[0,1]\smallsetminus\bigcup_{q>p}L_{q}[0,1]$ is spaceable for every $p>0$, Linear Algebra Appl. \textbf{436} (2012), 2963-2965.

%\bibitem{2} J. Carmona Tapia, J. Fernández-Sánchez, J. B. Seoane-Sepúlveda, W. Trutschnig, Lineability, spaceability, and latticeability of subsets of $C([0,1])$ and Sobolev spaces, Rev. R. Acad. Cienc. Exactas Fís. Nat. Ser. A Mat. RACSAM 116 (2022), no. 3, Paper No. 113, 20 pp. 46B87.

%\bibitem{LAA2014} K. C. Ciesielski, J. L. G\'amez-Merino, D. Pellegrino, J. B. Seoane-Sep\'ulveda, Lineability, spaceability, and additivity cardinals for Darboux-like functions, Linear Algebra Appl. \textbf{440} (2014), 307-317.

%\bibitem{10} K. C. Ciesielski, J. B. Seoane-Sepúlveda, A century of Sierpiński-Zygmund functions, Rev. R. Acad. Cienc. Exactas Fís. Nat. Ser. A Mat. RACSAM \textbf{113} (2019), no. 4, 3863-3901.

\bibitem{DR} D. Diniz, A. B. Raposo Jr., A note on the geometry of
certain classes of linear operators, Bull. Braz. Math. Soc. (N.S.) \textbf{52} (2021), 1073-1080.

%\bibitem{DFPR} D. Diniz, V. V. Fávaro, D. Pellegrino, A. B. Raposo Júnior, Spaceability of the sets of surjective and injective operators between sequence spaces, RACSAM. \textbf{114} (2020), no. 4, Paper No. 194, 11 pp.

\bibitem{fprs} V. F\'{a}varo, D. Pellegrino, A. B. Raposo Jr., G. S. Ribeiro, General criteria for a strong notion of lineability, arXiv:2303.01182 [math.FA].

\bibitem{FPR} V. F\'{a}varo, D. Pellegrino, P. Rueda, On the size of the set of unbounded multilinear operators between Banach spaces, Linear Algebra Appl. \textbf{606} (2020), 144-158.

\bibitem{FPT} V. F\'{a}varo, D. Pellegrino, D. Tomaz, Lineability and spaceability: a new approach, Bull. Braz. Math. Soc. \textbf{51} (2020), 27-46.

%\bibitem{GMSS} J. L. G\'amez-Merino, G. A. Mu\~noz-Fern\'andez, V. M. S\'anchez, J. B. Seoane-Sep\'ulveda, Sierpi\'nski-Zygmund functions and other problems on lineability, Proc. Amer. Math. Soc. \textbf{138} (2010), 3863-3876.


%\bibitem{gamezmunozseoane2010} J. L. G\'amez-Merino, G. A. Mu\~noz-Fern\'andez, J. B. Seoane-Sep\'ulveda, Lineability and additivity in $\mathbb{R}^\mathbb{R}$, J. Math. Anal. Appl. \textbf{369} (2010), 265-272.

\bibitem{GQ} V. I. Gurariy, L. Quarta, On Lineability of Sets of Continuous Functions, J. Math. Anal. Appl. \textbf{294} (2004), 62-72.

\bibitem{gurariy1966} V. I. Gurariy, Subspaces and bases in spaces of
continuous functions, Dokl. Akad. Nauk SSSR \textbf{167} (1966), 971-973.

%\bibitem{DR2} D. Kitson, R. M. Timoney, Operator ranges and spaceability, J. Math. Anal. Appl. \textbf{378} (2011), 680-686.

\bibitem{mpps2008} G. A. Mu\~noz-Fern\'andez, N. Palmberg, D. Puglisi, J. B. Seoane-Sep\'ulveda, Lineability in subsets of measure and function spaces, Linear Algebra Appl. \textbf{428} (2008), 2805-2812.

%\bibitem{PR} D. Pellegrino, A. B. Raposo Júnior, Pointwise lineability in sequence spaces, Indag. Math. (N.S.) \textbf{32} (2021), no. 2, 536-546.

%\bibitem{SS} J. B. Seoane-Sep\'ulveda, Chaos and lineability of pathological phenomena in analysis, Thesis (Ph.D.), Kent State University, ProQuest LLC, Ann Arbor, MI (2006).

%\bibitem{PS} D. Puglisi, J. B. Seoane-Sepúlveda, Bounded linear non-absolutely summing operators, J. Math. Anal. Appl. \textbf{338} (2008), 292-298.

%\bibitem{SR} C. P. Stegall, J. R. Retherford, Fully nuclear and completely nuclear operators with applications to $L^1$- and $L^\infty$-spaces, Trans. Amer. Math. Soc. \textbf{163} (1972), 457-492.	

%\bibitem{daniel} D. Tomaz, Linear structure in certain subsets of quasi-Banach sequence spaces, Linear Multilinear Algebra 67 (2019), no. 8, 1561-1566.
\end{thebibliography}
\end{document}